\documentclass[11pt,a4paper]{article}
\usepackage[T1]{fontenc}
\usepackage[utf8]{inputenc}
\usepackage{authblk}
\usepackage[all]{xy}
\usepackage{amsthm}
\usepackage{amsmath}
\usepackage{amssymb}
\usepackage{color}
\usepackage{float}
\usepackage{graphicx}
\usepackage[english]{babel}
\usepackage{verbatim}
\theoremstyle{plain}
\def\Rc{{\mathcal{R}}}

\def\Pb{{\mathbb P}}

\def\Zb{{\mathbb Z}}
\def\Cb{{\mathbb C}}
\def\Qb{{\mathbb Q}}

\def\Bc{{\mathcal B}}

\def\Gc{{\mathcal G}}

\def\Tc{{\mathcal T}}

\def\C~{\widetilde{C}}

\def\QP{{Q_{_P}}}

\def\Circlearrowright{\ensuremath{%
  \rotatebox[origin=c]{180}{$\circlearrowright$}}}
\newcommand{\q}[1]{``#1''}

\newtheorem{theorem}{Theorem}[section]
\newtheorem{definition}[theorem]{Definition}
\newtheorem{lemma}[theorem]{Lemma}
\newtheorem{example}[theorem]{Example}

\newtheorem{corollary}[theorem]{Corollary}
\newtheorem{solution*}{Solution}

\newtheorem{proposition}[theorem]{Proposition}
\newtheorem{remark}[theorem]{Remark}

\numberwithin{equation}{section}
\newcommand{\HOM}{\operatorname{Hom}}

\newcommand{\Proj}{\operatorname{Proj}}

\DeclareMathOperator{\Hom}{\mathscr{Hi}\text{\kern -3pt {\calligra\large om}}\,}

\DeclareMathOperator{\col}{col}

\title{Smooth Toric Quiver Varieties}
\author{Amir Nasr\thanks{amirnasr@ipm.ir}}
\affil{School of Mathematics, IPM, Tehran, Iran}

\date{\today}

\begin{document}
\maketitle
\begin{abstract}
We study smoothness of toric quiver varieties. When a quiver $Q$ is defined with the identity dimension vector,  the corresponding quiver variety is also a toric variety. So it has both fan representation and quiver representation. We work only on quivers with canonical weight and we classify smooth such toric quiver varieties. We show that a variety corresponding to a quiver with the identity dimension vector and the canonical weight is smooth if and only if it is a product of projective spaces or their blowups. 
\end{abstract}
\maketitle
\begin{section}{Introduction}
Toric varieties and quiver varieties are two types of varieties with combinatorial information. A toric variety can be constructed from a fan in a lattice and a quiver variety is the moduli space of representations of a quiver. In this project we work on varieties with both properties. Namely, if $Q=(Q_0, Q_1)$ is a quiver with the identity dimension vector, then the moduli space of representations of $Q$ which we call a quiver variety, has the structure of a toric variety.

Let $Q=(Q_0, Q_1)$ be a quiver with the identity dimension vector. Then to define the moduli space we need to consider a stability condition, called weight of the quiver, $w:Q_0\rightarrow \Zb$. The weight we work with is the most natural weight where we assign to each vertex the number of its in-coming arrows minus the number of its out-going arrows, which is what we call the canonical weight. Considering the identity dimension vector and the canonical weight, we classify smooth toric quiver varieties. 

In Section 2 we give an introduction about toric quiver varieties and we show how a quiver with the identity dimension vector has the structure of a toric variety. In Section 3 we define the correspondence between a fan and the quiver of a toric quiver variety. We will see that we can find a fan representation of the variety directly from the quiver and we do not need to find the variety first. In Section 4 we define simple quivers and we show how a quiver can be contracted to a simple one. Moreover, we will show that the toric quiver variety may not be changed or blown down after contracting a quiver to a simple quiver. In section 5 we classify smooth toric quiver varieties which is the main result of the paper. We first classify smooth toric quiver varieties coming from a simple quiver and then use the results of Section 4 to generalize the result and classify all smooth toric quiver varieties.
\end{section}
\begin{section}{Definitions and notations}\label{section2}
A quiver is a pair $Q=(Q_0, Q_1)$, where $Q_0=\{v_i\}_{1\leq i\leq n}$ is the set of vertices and $Q_1=\{v_i \xrightarrow{x}v_j\}_{i, j}$ is the set of arrows. The start point (vertex) and the end point (vertext) of an arrow are called the source and the target respectively. A path is a set of arrows and vertices such that we start from a vertex, go through the arrows and end at another vertex. A cycle is a path where each vertex and arrow is gone through only once and the start and end points are the same. A directed cycle is a cycle in which every arrow is passed from its source toward its target. In this paper we only work on quivers with no directed cycles. Between two vertices there might be one or more arrows, so a cycle can contain only two vertices and a pair of arrows between them, then of course all arrows between two vertices should be in the same direction. We call a cycle proper if it contains no more than one arrow between each pair of vertices. Two cycles are called edge connected if they share at least one arrow. By induction, a set of cycles is called edge connected if we can find one of them edge connected to the rest and the rest are also edge connected. 

We use a capital letter $X$ for a bunch of arrows between two vertices and the number of arrows in a bunch is called the multiplicity, $m_X$. If $m_X>1 (=1)$, then $X$ is called a multiple-arrow $($single-arrow$)$. 

Let $Q=(Q_0, Q_1)$ be a quiver. A dimension vector is a tuple $d=(d_i)_i\in\Zb^{Q_0}$. Then to each vertex $v_i$ a vector space $M_i$ of dimension $d_i$ is assigned and a representation of $Q$ is
\[
R_d=\bigoplus_{x:v_i\rightarrow v_j}\HOM(M_i, M_j).
\]

Also there is an algebraic group
\[
G_d=\prod_{v_i}GL(M_i)
\]
 acting on $R_d$ as $g.M:=(g_i)_i.(M_{\alpha})_{\alpha}=(g_iM_{\alpha}g_j^{-1})_{\alpha:v_i\rightarrow v_j}$. The $G_d$-orbits correspond bijectively to the isomorphism classes of representations of $Q$ \cite{R}.

 For the purpose of working on toric quiver varieties, in this paper we only work with the identity dimension vector, i.e.  $d(v)=1$ for any vertex $v$ (see more details in Remark \ref{torus}). Hille calls representations of such quivers \q{thin sincere representations}\cite{Hi}. Then we write $\Rc:=R_d=\bigoplus_{x\in Q_1}\Cb$ and  $\Gc:=G_d=(\Cb^*)^{|Q_0|}$. So a representation is in the form of the tuple $(\nu_x)_{x\in Q_1}\in \Cb^{|Q_1|}$. Therefore we can assign to each arrow $x$ a variable and if there is no ambiguity we show it by the same notation $x$.  Similarly, $X^{\alpha}$ may refer to a monomial which is multiplication of (non-negative) powers of all the correspondent variables.

To define the moduli space of representations of a quiver we consider a linear function $\theta:Q_0\rightarrow \Zb$ called the weight such that $\theta.d=0$. One perhaps say the most obvious weight is the canonical weight where $\theta(v)$ is the number of arrows with their target at $v$ minus the number of arrows with their source at $v$, so $\theta.d=\sum_{v\in Q_0}\theta(v)=0$. Considering the canonical weight, any thin sincere representation is $\theta$-stable, see \cite[Definition 1.3]{Hi}.

Read \cite{S} to see how the moduli space is defined for $\theta$-stable thin sincere representations for any weight $\theta$. Since we restrict our self to the canonical weight, we rephrase the definition as follows.
\begin{definition}\label{moduli space}
Let $Q=(Q_0, Q_1)$ be a quiver with the identity dimension vector and the canonical weight. Consider a function $w:Q_1\rightarrow \Zb^{|Q_0|}$ such that $w(v_i \xrightarrow{x}v_j)$ gets zeros everywhere except in the $i$-th and $j$-th positions which gets $-1$ and $1$ respectively. If $x$ refers to a variable, we define $w(x)$ the same as its corresponding arrow and for any other variable $y$, $w(xy):=w(x)+w(y)$. Then by induction $w(X^{\alpha})$ is defined for any monomial $X^{\alpha}$. Now a $\theta$-invariant monomial of $Q$ is a monomial $X^{\alpha}$ where $w(X^{\alpha})=\theta(Q_0)$. Then the moduli space is the variety
\[
X_Q=\Proj\Cb[X^{\alpha}|~~w(X^{\alpha})=\theta(Q_0)]
\]
\end{definition}
\begin{proposition}
Let $X_Q$ be the moduli space of $\theta$-stable thin sincere representation as defined in Definition \ref{moduli space}. Then $X_Q$ is a projective variety and the points are isomorphism classes of $\theta$-stable thin sincere representations.\cite{K}
\end{proposition}
\begin{remark}\label{torus}
$\Gc:=(\Cb^*)^{|Q_0|}$ is a torus acting on the affine space $\Rc=\bigoplus_{x\in Q_1}\Cb$. Since the scalar torus acts trivially on $\Rc$, we can consider $P\Gc=\Gc/\Cb^*$ and the orbits of this action corresponds bijectively to the isomorphism classes of $\theta$- stable this sincere representations. On the other hand, there exists the torus $(\Cb^*)^{|Q_1|}$ acting on the affine space $\Rc$ via multiplication on the left. Therefore, the torus $\Tc=(\Cb^*)^{|Q_1|}/P\Gc\cong(\Cb^*)^{|Q_1|-|Q_0|+1}$ acts on the space of isomorphism classes of $\theta$- stable this sincere representations and so, $\Tc$ acts on the moduli space $X_Q$. This makes the $X_Q$ be a toric variety.
\end{remark}
Note that since the weight we always work with is the canonical weight we may drop $\theta$ and simply write stable and invariant which mean $\theta$- stable and $\theta$- invariant receptively.
\end{section}
\begin{section}{Fan Representation of a Quiver}\label{section3}
In this section we explain how to find the fan representation of a toric quiver variety. Let $Q$ be a quiver with the identity dimension vector. Then $X_Q$ is a toric quiver variety and so it has a fan representation. We show how to construct the fan directly from the quiver. 
\begin{definition}
Consider the free abelian group over the set of arrows. By abuse of notation we write the bases of the free group by $\{x\in Q_1\}$. Then if $C$ is a path (cycle)  with arrows $x_1, \ldots, x_i$, we write it as $C:\pm x_1\pm\cdots\pm x_i$ where for each arrow the sign depends on its direction in the path $($cycle$)$. Two or more cycles are called independent if non of them can be written as an integer sum of the others. Once a single-arrow is replaced by a multiple-arrow with multiplicity $m$, then the multiple arrow adds $m-1$ independent cycles. A set of independent cycles which generates all cycles is called a basis of cycles and shown as $\Bc^{\tiny \Circlearrowright}$. Then the dimension of the quiver variety is the same as the size of  $\Bc^{\tiny \Circlearrowright}$.
\end{definition}
Read \cite{S} to see how the fan representation of a toric quiver varieties is defined. Here we bring a short description of the procedure.
\begin{definition}
Supporting quiver of a quiver $Q$, shown as $Q^s$, is a quiver with vertices $Q^s_0=Q_0$  and each multiple-arrow is replaced by an arrow in the same direction.
\end{definition}
Correspondent matrix of a quiver is an $m\times n$-matrix where $m$ is the number of arrows  and $n$ is the number of vertices. Each row represents an arrow and each column represents a vertex. In any row  there is only two non-zero entries, a $^{'}-1^{'}$ for the source and a $^{'}1^{'}$ for the target of the corresponding arrow, namely, for any arrow $x$, the row corresponding to $x$ is $w(x)$. Then a Gale Matrix of $M$, $M_G$, is a $k\times n$ matrix defined as follows.
\begin{definition}
Let $M$ be the correspondent matrix of a quiver $Q$. Columns of $M$ generate a sub-space of $\Qb^n$. Let $\{v_{i_1}, \ldots, v_{i_{n-k}}\}$ be a basis for $\col(M)$ and let $\{u_1, \ldots, u_k\}$ be an extension to a basis of ~$\Qb^n$; $\Bc=\{v_{i_1}, \ldots, v_{i_{n-k}}, u_1, \ldots, u_k\}$. Let $A$ be the transition matrix of changing the basis $\Bc$ to the standard basis, i.e. $[\Bc]A=I_n$. Now $M_G$ is the last $k$ rows of $A$.\\
\begin{proposition}
Then each arrow $x$ corresponds to a column of $M_G$ which represents  a ray, shown by $\rho_x$, in the fan representation of the toric quiver variety. Note that $M_G.M=0$ and since the fan representation is independent from the choose of $\Bc$, any $k\times n$ matrix $\Gamma$ with the property $\Gamma.M=0$ works. \cite{S}
\end{proposition}
We define dot product of two rays $\rho_x.\rho_y$ to be dot product of their corresponding columns in $M_G$. This is used later in Proposition \ref{90degree} when we talk about the position of a ray compared to another ray.
\end{definition}
\begin{proposition}\label{vertex degree vs rays}
Let $Q$ be a quiver with no directed cycle and let $M_G$ be a Gale matrix. Consider a vertex $q$ and its out-going  and in-coming arrows $x_1, \ldots, x_i$ and $y_1, \ldots, y_j$ respectively. If we show the corresponding columns in $M_G$ by $\mathcal{X}_1, \ldots, \mathcal{X}_i$ and $\mathcal{Y}_1, \ldots, \mathcal{Y}_j$, then $\mathcal{X}_1+\cdots + \mathcal{X}_i=\mathcal{Y}_1+\cdots +\mathcal{Y}_j$.
\end{proposition}
\begin{proof}
Let $q$ be a vertex of $Q$ and let $x_1, \ldots, x_i$, $y_1, \ldots, y_j$, $\mathcal{X}_1, \ldots, \mathcal{X}_i$ and $\mathcal{Y}_1, \ldots, \mathcal{Y}_j$ be as above. Consider the column of $M$ corresponding to $q$ and show it by $v_q$. Then by the definition, we have $M_G.M=0$ and so $M_G.m_q=0$ which is what we are looking for.
\end{proof}
\end{section}
\begin{section}{Simple Quivers}\label{section4}
\begin{definition}
Let $v_i \xrightarrow{x}v_j$ be a single-arrow and let $d_i=d_j$. By contracting $x$ we mean removing $x$ and unifying $v_i$ and $v_j$ to a single vertex, call it $v_{ij}$. Then $d_{ij}:=d_i=d_j$ and $\theta_{ij}:=\theta_i+\theta_j$. Note that $\theta_{ij}$ satisfies the definition of a canonical weight at the vertex $v_{ij}$ and therefore the new dimension and weight  vectors are the identity dimension vector and the canonical weight. We show the new quiver with $Q_x$.
\end{definition}
\begin{definition}
A single-arrow $v_i \xrightarrow{x}v_j$    is called contractible if contracting $x$ causes no directed cycle. We call a quiver $Q$ simple if it is connected and has no contractible arrow. Obviously any connected quiver can be contracted to a simple quiver.
\end{definition}
\begin{lemma}
Let $x$ be a contractible arrow. There is a basis of cycles, $\Bc^{\tiny \Circlearrowright}_x$, such that any arrow is either same-directed with $x$ in all cycles of $\Bc^{\tiny \Circlearrowright}_x$ or opposite-directed in all of them.
\end{lemma}
\begin{proof}
Let there exist two cycles $C_1$ and $C_2$ containing $x$ and a single-arrow $y$ such that $y$ is same-directed with $x$ in $C_1$ and opposite directed in $C_2$.
\begin{figure}[h!]\begin{center}
  \includegraphics[scale=0.4]{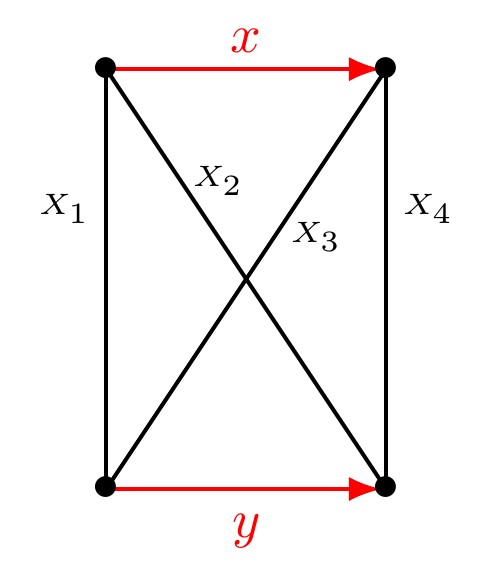}
  \label{fig:M2}
\end{center}
\end{figure}
\begin{comment}
\[
\begin{tikzpicture}[->,>=stealth',shorten >=1pt,auto,node distance=5cm, thick,main node/.style={circle,fill=blue!20,draw,font=\sffamily\Large\bfseries}]    
\node (a) at (0,0) {};
\node (b) at (2,0) {};
\node (c) at (0,-3) {};
\node (d) at (2,-3) {};
\begin{scope}[->,>=Latex]
   %\draw[-{>[scale=1.2]}, double] (a) to (c) ;
   \draw[color=red, ->] (a.center)  to (b.center) ;
   \draw[-] (a.center)  to (c.center) ;
   \draw[-] (a.center)  to (d.center) ;
   \draw[-] (b.center)  to (c.center) ;
   \draw[-] (b.center)  to (d.center) ;
   \draw[color=red, ->] (c.center)  to (d.center) ;
   %\draw[dotted, -] (d.south west) [out=40,in=140] to (e.south east) ;

\node[color=red] (alpha) at (1,0.2) {$x$};
\node[color=red] (beta) at (1,-3.3) {$y$};
\node (x1) at (-0.3,-1) {${\scriptscriptstyle X_1}$};
\node (x2) at (0.8,-0.8) {${\scriptscriptstyle X_2}$};
\node (x3) at (1.5,-1.2) {${\scriptscriptstyle X_3}$};
\node (x4) at (2.3,-1) {${\scriptscriptstyle X_4}$};

\node (a) at (0,0) {$\bullet$};
\node (b) at (2,0) {$\bullet$};
\node (c) at (0,-3) {$\bullet$};
\node (d) at (2,-3) {$\bullet$};
\end{scope}
\end{tikzpicture}
\]
\end{comment}
Where each $X_i$ is a paths containing multiple-arrows and single-arrows. Then $C_1:x\pm X_3+y\pm X_2$ and $C_2:x\pm X_4-y \pm X_1$. We also have the cycles $C_3:x\pm X_3\pm X_1$ and $C_4:x\pm X_4\pm X_2$. Now we can replace $C_2$ by $C_3+C_4-C_1$. If $Y$ is a multiple-arrow of $Y=\{y_1, \ldots, y_n\}$, we do everything the same for $Y$, then in the basis we replace $Y$ by each arrow $y_i$ separately.
\end{proof}

\begin{lemma}\label{product of vertices}
\cite[Exercise II.5.11]{Ha} Let $S$ and $T$ be two graded rings with $S_0=T_0=\Cb$. Their Cartesian product $S\times_{\Cb}T$ is the graded ring $\bigoplus_{d\geq0}S_d\otimes_{\Cb}T_d$. Then $\Proj(S\times_{\Cb}T)\cong\Proj(S)\times_{\Cb}\Proj(T)$.
\end{lemma}

\begin{proposition}\label{X_Q_P}
Let $Q$ be a quiver with no directed cycle and let $P$ be a sub-quiver of $Q$ with adjusted weights. Consider the quiver $Q$ with all its vertices contained in $P$ contracted to one single vertex $p$ and call it by $\QP$. If $p$ is not contained in any proper cycle of $\QP$, then $X_Q$ is isomorphic to the product of $X_P$ and  $X_\QP$.
\end{proposition}
\begin{proof}
Arrows of $Q$ starting at a vertex of $P$ and ending at a vertex not in $P$ are transformed to arrows in $Q_{_P}$ starting at $p$. Similarly, arrows of $Q$ ending at a vertex of $P$. Let $\{X_1,\ldots, X_k\}$ be the set of  arrows starting/ending at $p$ and let $m_1, \ldots, m_k$ be their multiplicities. If $X^{\alpha}$ is an invariant monomial of $Q$, we want to show that the degree of each $X_i$ in $X^{\alpha}$ is exactly $m_i$. To do so, without loose of generality we show it for only $i=1$. We also assume that $p$ is the source of $X_1$. Let $p$ be connected to a quiver $Q_1$ and let $\alpha_1$ be the degree of $X_1$ in $X^{\alpha}$. If $Q_1$ is one single vertex $q$, then the weight of $q$ is $m_1$ and so $\alpha_1=m_1$. Let $Q_1$ have more than one vertex and connected to $p$ by the vertex $q$.

In $X^{\alpha}$, we change the degrees of all variables equal to zero except the ones corresponding to arrows in $Q_1$ and call it $X^{\alpha^{'}}$. On the other hand, consider $X^{\beta}$ where the degree of each variable is one and in a similar way find $X^{\beta^{'}}$. Since both $X^{\alpha^{'}}$ and $X^{\beta^{'}}$ are representations of $Q_1$, the weight adjustment should be the same in both of them. So, $\alpha_1=m_1$.

Now let $X^{\alpha}$ be an invariant monomial of $Q$. If we change the degrees of all variables corresponding to arrows of $P$ equal to zero and show it by $X^{\beta}$, then $X^{\beta}$ is an invariant monomial of $Q_{_P}$ and if we change the degrees of all variables corresponding to arrows of $Q\backslash P$ equal to zero and show it by $X^{\beta'}$, then $X^{\beta'}$ is an invariant monomial of $P$ and $X^{\alpha}=X^{\beta}\otimes X^{\beta'}$. On the other hand, let $X^{\beta}$ and $X^{\beta'}$ be invariant monomials of $Q_{_P}$ and $P$ respectively, then $X^{\beta}\otimes X^{\beta' }$ is an invariant monomial of $Q$. Thus by Lemma  \ref{product of vertices} $\Cb[R_Q]$ is the Cartesian product $\Cb[R_P]\times_{\Cb} \Cb[R_{Q_{_P}}]$ and 
\[
X_Q=\Proj(\Cb[R_Q])\cong\Proj(\Cb[R_P]\times_{\Cb} \Cb[R_{Q_{_P}}])=X_\QP\times X_P.
\]
\end{proof}

\begin{proposition}\label{90degree}
Let $x$ be a contractible arrow in a quiver $Q$. If there is an arrow $y$ in a cycle with the same direction as $x$, then the dot product $\rho_{x}.\rho_{y}$ is non-negative.
\end{proposition}

\begin{proof}
The number of independent cycles, say $k$, is the same as the dimension of the variety. Let also $m$ be the number of arrows. First we consider the case where $y$ is a single-arrows. We will define a $k\times m$ Gale matrix in which the dot product $\rho_{x}.\rho_{y}$ is non-negative. Per cycle in $\Bc^{\tiny \Circlearrowright}_x$ we write the row of the coefficients of all the arrows, considering that the coefficient of an arrow which is not in a cycle is zero. Assuming that an arrow $y$ is in a cycle with $x$, then their coefficients have always the same sign if not zero. This guarantees that  $\rho_x.\rho_y\geq 0$. Now if there is a multiple-arrow, $y=\{y_1, \ldots, y_l\}$, then $\Bc^{\tiny \Circlearrowright}_x$ has $l-1$ cycles more compared to the case that $y$ is a single-arrow. Namely, $y$ in each cycles  shall be replaced by any $y_j$. So, based on the discussion above, $\rho_x.\rho_{y_j}\geq0$.
\end{proof}

\begin{theorem}\label{X_Q/alpha}
Let $Q$ be a quiver with no directed cycle and a contractible arrow $x$. Then either $X_Q=X_{Q_x}$ or $X_{Q_{x}}$ is a blow down of $X_Q$.
\end{theorem}
\begin{proof}
If $x$ is not involved in any cycle, then $X_Q=X_{Q_x}$ by Proposition \ref{X_Q_P}. So we assume that $x$ is in a cycle. We first show that contracting $x$ is the same as removing a ray in the fan representation of the quiver variety. Let $Q$ be a quiver and let $M$ be the correspondent matrix. Let us contract a contractible arrow $v_i\xrightarrow{x}v_j$. Then in the matrix $M$, two columns of $M$, corresponding to $v_i$ and $v_j$, are added together and written as one column. Also the row corresponding to $x$, say the $l$-th row and show by $m_x$, will be entirely zero and removed. In the new correspondent matrix, show it by $M_x$, we write the new column as $v_{ij}$.

\begin{figure}[h!]\begin{center}
  \includegraphics[scale=0.6]{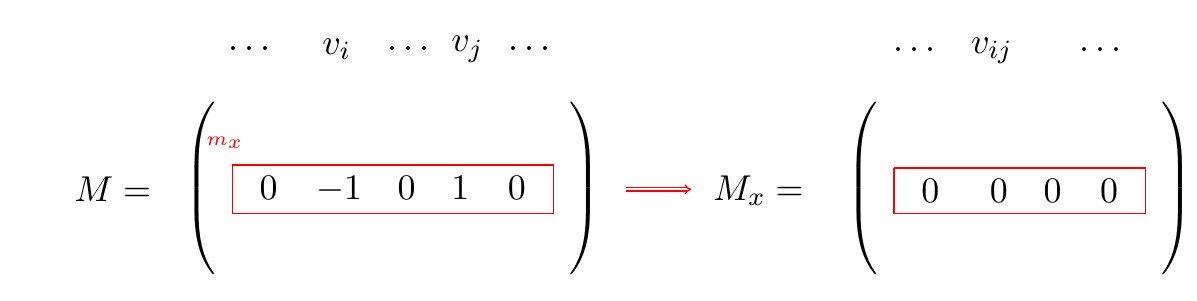}
  \label{fig:M2}
\end{center}
\end{figure}
\begin{comment}
\begin{tikzpicture}
         \matrix [matrix of math nodes,left delimiter=(,right delimiter=)] (m)
       {
           ~ &~&~&~&~ \\   
           ~~~~ &&&& ~~~~\\               
           0 &-1&~0~&1& 0 \\               
           ~~~~ &&&&~~~~ \\                   
            &&~&& \\           
        }; 
        \draw[color=red] (m-2-1.south west) -- (m-2-5.south east) -- (m-4-5.north east) -- (m-4-1.north west) -- (m-2-1.south west);
       % 
\node[] at (-3,0)   (a) {$ M=$};
\node[color=red] at (-1.8,0.5)   (a) {${\scriptscriptstyle m_x}$};
\node[] at (-0.6,1.5)   (a) {$v_i$};
\node[] at (0.8,1.5)   (a) {$v_j$};
\node[] at (0.2,1.5)   (a) {$\cdots$};
\node[] at (1.5,1.5)   (a) {$\cdots$};
\node[] at (-1.5,1.5)   (a) {$\cdots$};
\draw[color=red,double,implies-](3.2,0) -- (2.5,0);
  \end{tikzpicture}
  \begin{tikzpicture}
         M= \matrix [matrix of math nodes,left delimiter=(,right delimiter=)] (m)
          {
           ~ &~&~ & \\   
           ~~~~ &~&~& ~~~~\\               
           0 &~0~& 0 & 0 \\               
           ~~~~ &~&~&~~~~ \\                   
            ~&~&~&~ \\           
        };  
        \draw[color=red] (m-2-1.south west) -- (m-2-4.south east) -- (m-4-4.north east) -- (m-4-1.north west) -- (m-2-1.south west);

\node[] at (-2.8,0)   (a) {$ M_x=$};
\node[] at (-0.3,1.5)   (a) {$v_{ij}$};
\node[] at (-1.1,1.5)   (a) {$\cdots$};
\node[] at (0.9,1.5)   (a) {$\cdots$};
  \end{tikzpicture}
\end{comment}

Now we show that in $M_G$ all columns are the same as $(M_x)_G$ except the $l$-th column which is removed. First, note that in each row of a correspondent matrix, there is a $'-1'$ and a $'1'$ and adding all columns of $M$ (or $M_x$) gives a zero column. So at least one of the columns (any of them) is in the column space of the others. For $M$ we make sure that exactly one of $v_i$ or $v_j$ is in the basis $\Bc$, say $v_i$ and we set it to be $m_{i_1}$. Also, $c_1, \ldots, c_k$ can be chosen such that they all have zeros in the $l$-th component. Therefore, $[E_k]_{\Bc}$ has a non-zero weight for $v_i$ only if $k=l$. Namely, for $k\neq l$, $E_k$ is represented as a linear combination of columns except $v_i$. So, disregarding $v_{ij}$ for finding $(M_x)_G$ ensures that all columns of $M_G$ remains the same except the $l$-th column which is removed.\\
Now we only need to show that the ray corresponding to the removed ray is the sum of one or more rays corresponding to some other rays. For the rays corresponding to arrows we use the same notation as used for the arrows. Since $x$ is in a cycle we assume there are some arrows connected to $x$ at the target of $x$, $v_j$. Since $x$ is contractible , in any cycle that $x$ is involved, there is one or more arrow(s) in the direction of $x$. Then by Proposition \ref{vertex degree vs rays} the ray $x$ is the sum of the rays corresponding to some of these same-directed arrows. Using Proposition \ref{90degree} we can choose rays which are neighbours to $x$. Thus by \cite[Definition 3.3.17]{CLS} we get a blowdown.
\end{proof}

\end{section}

\begin{section}{Smooth Toric Quiver Varieties}\label{section5}

\begin{proposition}\label{cycle singular}
Let $Q$ be a simple quiver with no directed cycle, the identity dimension vector and the canonical weight. Then if $Q$ has a (non-directed) cycle, $X_Q$ is singular.
\end{proposition}

Before proving the lemma, let us work on the following example.

\begin{example}
Consider the following cycle. 
\begin{figure}[h!]
\begin{center}
  \includegraphics[scale=0.3]{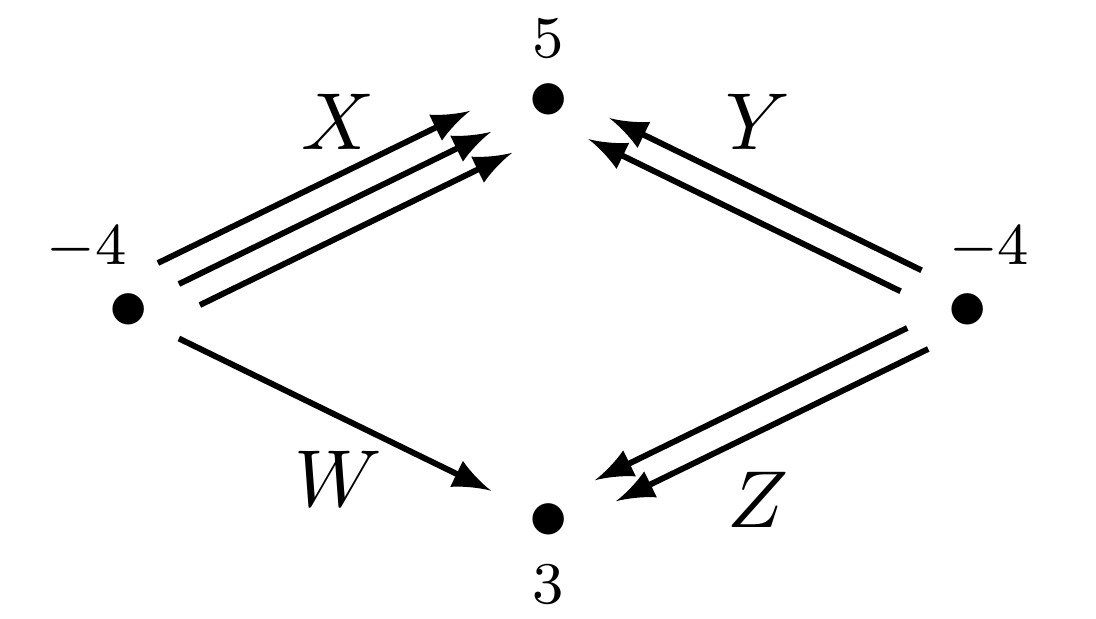}
  \label{fig:M1}
\end{center}
\end{figure}

Let $M=X^kY^lZ^mW^n$ be an invariant of the quiver where any of the  capital letters in $M$, like $X$, corresponds to each of $x_1, x_2$, or $x_3$. Here $X$ and $Z$ are in the same direction and $Y$ and $W$ in the same opposite direction and adding the power of each of the variables, results in increasing the power of the other variable in the same direction and decreasing the power of the variables in the opposite direction. Since the power of non of the variables can be negative, the power of $X$ and $Y$ are both at least $1$. So, all invariants of the quiver are
\[
X^4YZ^3, X^3Y^2Z^2W, X^2Y^3ZW^2, XY^4W^3
\]

Note that they are all in the kernel of the quiver matrix
\[
\begin{bmatrix}
-1&0&0&-1\\
1&1&0&0\\
0&-1&-1&0\\
0&0&1&1
\end{bmatrix}.
\]

Since $XY$ is the common factor of all invariants, we factor $XY$ and we have $(XZ)^3, (XZ)^2YW, XZ(YW)^2, (YW)^3$ which correspond to the quiver 

\begin{figure}[h!]\begin{center}
  \includegraphics[scale=0.3]{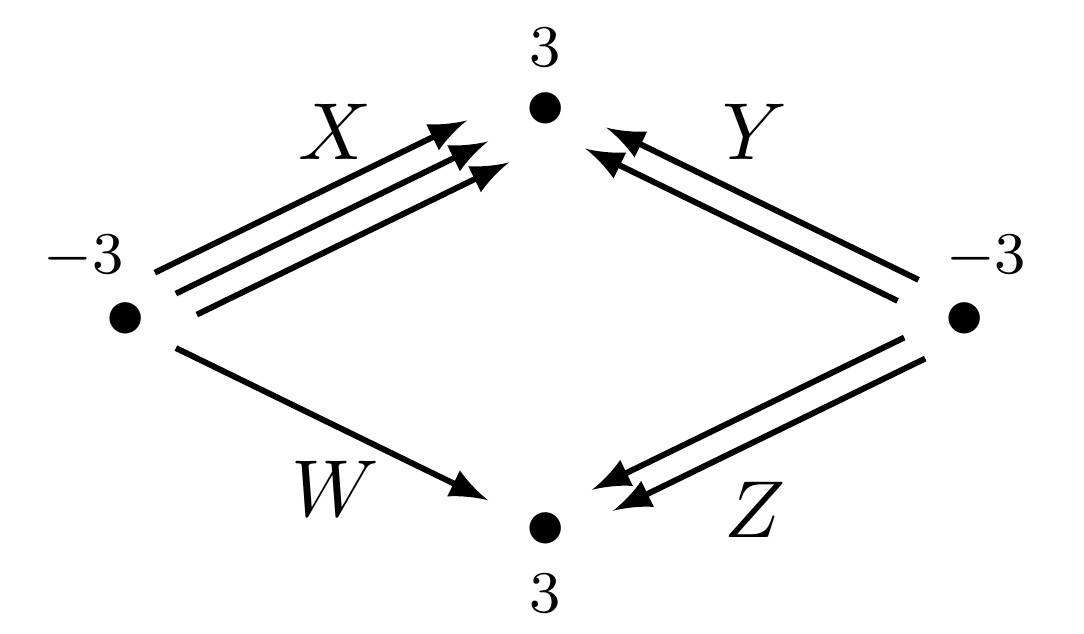}
  \label{fig:M2}
\end{center}
\end{figure}

\begin{comment}
\begin{equation}\label{fig:M2}
\begin{tikzpicture}[->,>=stealth',shorten >=1pt,auto,node distance=5cm, thick,main node/.style={circle,fill=blue!20,draw,font=\sffamily\Large\bfseries}]
\node (a) at (0,0) {$\bullet$};
\node (b) at (4,0) {$\bullet$};
\node (c) at (2,1) {$\bullet$};
\node (d) at (2,-1) {$\bullet$};    
\node (y1) at (3,0.9) {$Y$};    
\node (y1) at (1,0.9) {$X$};     
\node (y1) at (3,-0.9) {$Z$};    
\node (x1) at (1,-0.8) {$W$};   
\node (0) at (2,1.3) {${\scriptstyle 3}$};
\node (0) at (2,-1.3) {${\scriptstyle 3}$};
\node (w) at (4.1,0.3) {${\scriptstyle -3}$};
\node (v) at (-0.2,0.3) {${\scriptstyle -3}$};  
\begin{scope}[->,>=latex]
 \foreach \i in {-1,0,1}{% 
   \draw[->] ([yshift=(\i-1) * 0.1 cm, xshift=-\i * 0.1 cm]a.north east) to ([yshift=(\i+1) * 0.1 cm, xshift=-\i * 0.1 cm]c.south west) ;}
 \foreach \i in {-0.5,0.5}{% 
 \draw[<-] ([yshift=(\i+1)* 0.1 cm, xshift=\i * 0.1 cm]c.south east) -- ([yshift=(\i-1) * 0.1 cm, xshift=\i * 0.1 cm]b.north west) ;}
 \foreach \i in {0}{% 
   \draw[->] ([yshift=(\i+1) * 0.1 cm, xshift=-\i * 0.1 cm]a.south east) to ([yshift=(\i-1) * 0.1 cm, xshift=-\i * 0.1 cm]d.north west) ;}
 \foreach \i in {-0.5,0.5}{% 
 \draw[->] ([yshift=(\i+1)* 0.1 cm, xshift=-\i * 0.1 cm]b.south west) -- ([yshift=(\i-1) * 0.1 cm, xshift=-\i * 0.1 cm]d.north east) ;}
\end{scope}
\end{tikzpicture}
\end{equation}
\end{comment}

Dividing all the weights with $3$, we only get the invariants $XZ, YW$. Then the corresponding variety is
\[
X=\Proj\left(\frac{\Cb[x_1z_1: x_1z_2: x_2z_1: x_2z_2: x_3z_1: x_3z_2: y_1w: y_2w]}{\text{corresponding relations}}\right)
\]
which is singular. Now, going back to the quiver in \ref{fig:M2}, we get the Veronese map $\nu_3:\Pb^7\rightarrow\Pb^{119}$ and  Veronese isomorphism $X\stackrel{\nu_3}{\cong}\nu_3(X)=X'$, so $X'$  is singular. Now, putting the $XY$ factor back in the invariants, we get the Segre map $\sigma:\Pb^{119}\times\Pb^5\rightarrow\Pb^{719}$ and Segre isomorphism $X'\stackrel{\sigma}{\cong}\sigma(X')=X_Q$, where $\Pb^5$ comes from $\Cb[x_1y_1:x_2y_1:x_3y_1:x_1y_2:x_2y_2:x_3y_2]$ for the $XY$ factor. So $X_Q$ is singular.
\end{example}

\begin{proof}[Proof of Proposition \ref{cycle singular}]
By Proposition \ref{X_Q_P}, we only need to work on a set of edge connected cycles. So, consider a cycle $Q_C$. Since it is not an oriented cycle, it has arrows in both directions, clockwise and anti-clockwise. We name all clockwise arrows with $\{X_i\}_{0\leq i\leq n}$ and anti-clockwise ones with $\{Y_j\}_{0\leq j\leq m}$. Let $\alpha_i$ and  $\beta_j$ show the number of arrows in the multi-arrows $X_i$ and $Y_j$ respectively. Also we show the arrows of $X_i$ and $Y_j$ by $x_{i, t_i}$ and $y_{j, s_j}$. Without loss of generality, assume that $\alpha_1$ and $\beta_1$ are the smallest numbers among $\alpha_i$ and $\beta_j$ respectively. Now let $X=X_1^{k_1}\ldots X_n^{k_n}.Y_1^{l_1}\ldots Y_m^{l_m}$ be an invariant of the cycle. Then for each $0\leq i\leq n$ and $0\leq j\leq m$, $k_i\geq \alpha'_i:=\alpha_i-\alpha_1$ and $l_j\geq\beta'_j:=\beta_j-\beta_1$. So, from all invariants we can factor $X':=X_1^{\alpha'_1}\ldots X_n^{\alpha'_n}.Y_1^{\beta'_1}\ldots Y_m^{\beta'_m}$  and  we adjust the weights of the vertices. Therefore in the first step, we will get a quiver with the property that the weight of each vertex is $\omega:=\alpha_1+\beta_1$, $-\omega$, or $0$ if it is a source, a sink, or neither respectively. In the second step, dividing all the weights by $\omega$, we get the invariants $X_1\ldots X_n$ and $Y_1\ldots Y_n$. Expanding them we get the following variety
\[
\Proj\left(\frac{\Cb[x_{1, t_1}\ldots x_{n, t_n}:y_{1, s_1}\ldots y_{m, s_m}]}{\text{corresponding relations}}\right)~~~~\begin{array}{c}
1\leq t_i\leq\alpha_i\\
1\leq s_j\leq\beta_j \end{array}.
\]
which is singular.  Now, going back to the quiver created in the first step, we get a Veronese map $\nu_{\omega}:\Pb^{\alpha+\beta-1}\rightarrow\Pb^{ {{\omega+\alpha+\beta-1}\choose {\omega}}-1}$ and  Veronese isomorphism $X\stackrel{\nu_{\omega}}{\cong}\nu_{\omega}(X)=X'$, so $X'$  is singular. Now, putting the $X'$ factor back in the invariants, we get a Segre embedding  which is singular.

Now let $Q$ be a simple quiver with a set of edge connected cycles. Let $C_Q$ represent a cycle which is by itself simple. Then $C_Q$ is singular as a cycle by its own (with weights adjusted). Now we show that $X_Q$ is singular using a singular sub-variety, i.e. $X_{C_Q}$. To do so, let $A$ be the set of all invariants of $Q$ and let  $M$ be an invariant of $Q$. If the degree of each letter which corresponds to an arrow in $Q\backslash C$ is equal to the multiplicity of that arrow, we put $M$ in a set $B$. Now $X_{C_Q}\cong\Proj\left(\frac{\Cb[M\in B]}{f_1, \ldots, f_k}\right)$ and  is a singular sub-variety of  $X_Q\cong\Proj\left(\frac{\Cb[M\in A]}{f_1, \ldots, f_k, g_1, \ldots, g_l}\right)$ by Lemma 4.4. Let $p$ be a point of singularity of $X_{C_Q}$. Then the Jacobian matrix of $X_{C_Q}$ has a lower rank at $p$ than almost all other points. Then the Jacobian matrix of $X_Q$ has also a lower rank at $p$ than almost all other points.
\end{proof}

\begin{lemma}
Let $Q$ be a quiver and with no proper cycle. Then $X_Q$ is direct product of projective spaces.
\end{lemma}

\begin{proof}
It is obvious using Proposition \ref{X_Q_P}.
\end{proof}

\begin{lemma}
Let $Q$ be a quiver with a contractible arrow $x$. If $X_Q$ is smooth, then its blow down $X_{Q_x}$ is smooth too.
\end{lemma}

\begin{proof}
Let  $X_Q$ be smooth and $(N, \Delta)$ be its fan representation. Then each cone is generated by a set of rays which are basis for $N$. Now let $\sigma_1, \ldots, \sigma_m$ be the cones which contain $\rho_x$. Each of these cones is generated by a set of rays $($including $\rho_x)$ which are a basis for $N$. Since $\rho_x$ is a sum of rays in $\sigma_1, \ldots, \sigma_m$, in the fan $(N, \Delta_{\rho_x})$ with the ray corresponding to $x$ removed, the cone $\sigma=\bigcup_{i=1}^m\sigma_i$ is a smooth cone. See \cite[section 2.1]{F}
\end{proof}

\begin{corollary}
Let $X$ be a toric quiver variety. Then if $X$ is smooth, it is a direct product of projective spaces or blowups of such varieties. 
\end{corollary}
\end{section}

\newpage
~

\end{document}